\documentclass[12pt,reqno]{amsart}

\usepackage{tikz}
\usetikzlibrary{positioning, arrows.meta}
\usepackage{epsfig}
\usepackage{amscd}
\usepackage[mathscr]{eucal}
\usepackage{amssymb}
\usepackage{amsxtra}
\usepackage{amsmath}
\usepackage[all]{xy}
\usepackage{mathtools}
\usepackage[pdfencoding=auto]{hyperref}
\usepackage{bookmark}
\usepackage{caption}
\usepackage{colortbl}
\usepackage{arydshln}

\DeclarePairedDelimiter\abs{\lvert}{\rvert}

\theoremstyle{plain}
\newtheorem{thm}{Theorem}[section] 
\newtheorem{cor}[thm]{Corollary} 
\newtheorem{prop}[thm]{Proposition} 

\newtheorem*{conj}{Conjecture}

\theoremstyle{definition}

\theoremstyle{remark}

\begin{document}

\bigskip

\title[Approx. of Hardy $Z$-fun. via high-order sections]{On the approximation of the Hardy $Z$-function via high-order sections} 
\date{\today}

\author{Yochay Jerby}

\address{Yochay Jerby, Faculty of Sciences, Holon Institute of Technology, Holon, 5810201, Israel}
\email{yochayj@hit.ac.il}


%
%
% ...
\begin{abstract}
Sections of the Hardy $Z$-function are given by $Z_N(t) := \sum_{k=1}^{N} \frac{cos(\theta(t)-ln(k) t) }{\sqrt{k}}$ for any $N \in \mathbb{N}$. Sections approximate the Hardy $Z$-function in two ways: (a) $2Z_{\widetilde{N}(t)}(t)$ is the Hardy-Littlewood approximate functional equation (AFE) approximation for $\widetilde{N}(t) = \left [ \sqrt{\frac{t}{2 \pi}} \right ]$. (b) $Z_{N(t)}(t)$ is Spira's approximation for $N(t) = \left [\frac{t}{2} \right ]$. Spira conjectured, based on experimental observations, that, contrary to the classical approximation $(a)$, approximation (b) satisfies the Riemann Hypothesis (RH) in the sense that all of its zeros are real. We present theoretical justification for Spira's conjecture, via new techniques of acceleration of series, showing that it is essentially equivalent to RH itself.
 
\end{abstract}

\maketitle
%
%

\section{Introduction}

\subsection{Riemann's Analytic Extension of Zeta} The Riemann zeta function is given by $\zeta(s) := \sum_{n=1}^{\infty} n^{-s}$ in the range $Re(s)>1$. In his revolutionary 1859 work, Riemann extended the zeta function $\zeta(s)$ analytically to a meromorphic function on the entire complex plane with  a single pole at $s=1$. This extension allowed him to explore the zeta function from a complex analytic perspective, uncovering its profound connection to the distribution of prime numbers. Riemann's analytic continuation of $\zeta(s)$ is given by the integral representation:

\begin{equation}
\label{eq:int}
\zeta(s) = \frac{\Gamma(1-s)}{2 \pi i} \int_{- \infty}^{\infty} \frac{(-x)^s}{e^x-1} \frac{dx}{x}.
\end{equation}

Although the integral representation outlined in equation \eqref{eq:int} was of tremendous importance for Riemman's theoretical explorations in his manuscript, it is not amenable for direct computations.  To practically compute $\zeta(s)$, asymptotic techniques are required.

\subsection{The Approximate Functional Equation (AFE) and the Riemann-Siegel Formula} The development of the AFE for the $Z$-function due to Hardy and Littlewood, marks a cornerstone in computation methods of zeta \cite{HL,HL2,HL3,HL4}. The Hardy $Z$-function, denoted as \(Z(t)\), is the real function defined by
\begin{equation} \label{eq:Hardy}
Z(t) = e^{i \theta(t)} \zeta \left ( \frac{1}{2} +it \right )
\end{equation} 
where \(\theta(t)\) is the Riemann-Siegel \(\theta\)-function, given by the equation
\begin{equation}
\theta(t) = \text{arg} \left ( \Gamma \left ( \frac{1}{4} + \frac{i t}{2} \right ) \right ) -\frac{t}{2} \log(t),
\end{equation} 
see \cite{E,I}. The Hardy-Littlewood approximation is encapsulated in the formula
\begin{equation} 
\label{eq:HL} 
Z(t) =  2 \sum_{k=1}^{\widetilde{N}(t)} \frac{cos(\theta(t)-ln(k) t) }{\sqrt{k}}+R(t),
\end{equation}  
where $\widetilde{N}(t) = \left [ \sqrt{ \frac{ t}{2 \pi}} \right ]$ and the error term is given by 
\begin{equation} 
R(t)=O \left ( \frac{1}{\sqrt[4]{t}}  \right ).
\end{equation}
This representation, while powerful, for many purposes, requires further refinement of the error term $R(t)$ to enhance the level of precision.

 In the 1930's, C. L. Siegel uncovered previously unpublished notes by Riemann which revealed that, remarkably, Riemann not only knew the Hardy-Littlewood formula \eqref{eq:HL} but also developed complex saddle point techniques which enable the required further evaluation of its error term $R(t)$ for any order \cite{E,Si}. For instance, expanded to first-order, the Riemann-Siegel formula gives 
\begin{equation} 
R(t) = (-1)^{N(t)-1} \left ( \frac{t}{2 \pi} \right )^{-\frac{1}{4}}  \frac{cos ( 2 \pi (p^2-p -\frac{1}{16})) }{cos (2 \pi p ) }+ O \left ( \frac{1}{t^{\frac{3}{4}}} \right ),
\end{equation}
where $p= \sqrt{\frac{t}{2 \pi}} - \widetilde{N}(t)$. As an asymptotic formula, for a given specific value of $t$, indefinitely increasing the order of the expansion does not guarantee improved approximation. It is conjectured by Berry and Keating \cite{BeKe} that, when expanded to its optimal order, the Riemann-Siegel formula can reach accuracy level of exponentially decaying error $O \left ( e^{-\pi t} \right )$.

Since its introduction, the Riemann-Siegel formula, especially when enhanced with the Odlyzko-Schönhage algorithm \cite{OdlyzkoSchoenhage1988},  has been the main method for numerical verification of the RH, see \cite{Gourdon2004,Le,Od,Odlyzko1992,Ro,Titchmarsh1935,Turing1953}. The most comprehensive verification to date, achieved by Platt and Trudigian, confirms the RH up to $3\cdot 10^{12}$, see \cite{PT}. However, it can be argued that one of the great challenges in the Riemann hypothesis is the fact that while the Riemann-Siegel formula gives efficient methods for the numerical estimation of $R(t)$ to any given order, it offers little analytical insight on its structure, crucial for the understanding of the properties of the zeros in general, especially as the order of its expansion increases due to the complexity of the expressions involved.

\subsection{Sections of the $Z$-function and Spira's Approximation} In \cite{SP2, SP1} Spira introduced the notion of sections of the AFE of $Z(t)$ which are defined by 
\begin{equation} 
\label{eq:sections} 
Z_N(t) := \sum_{k=1}^{N} \frac{cos(\theta(t)-ln(k) t) }{\sqrt{k}},  
\end{equation}  
and conducted a study of their zeros for any $N \in \mathbb{N}$. In particular, Spira noted that sections of $Z(t)$ give an additional approximation of the function to that of the AFE 
\begin{equation} 
\label{eq:Spira}
Z(t) = Z_{N(t)} (t) + O \left ( \frac{1}{\sqrt[4]{t}} \right ),  
\end{equation}  
at the higher range $N(t) = \left [ \frac{t}{2} \right ]$. This approximation and its proof most certainly pre-dates Spira and was probably already known to Riemann, see Theorem 1.8 of \cite{I2} for a variant. A particularly noteworthy aspect of this additional approximation \eqref{eq:Spira}, however, is its unique properties when compared to the classical AFE, discovered by Spira in his experimental investigations. For instance, Fig. \ref{fig:f1.0} shows the values of the sections $Z_N(t)$ for the fixed value of $t=3000$ and $N=1,...,1500$ (blue), the Hardy-Littlewood approximation of $\frac{1}{2}Z(3000)$ (green) and Spira approximation $Z(3000)$ (brown):  
\begin{figure}[ht!]
\centering
\includegraphics[scale=0.325]{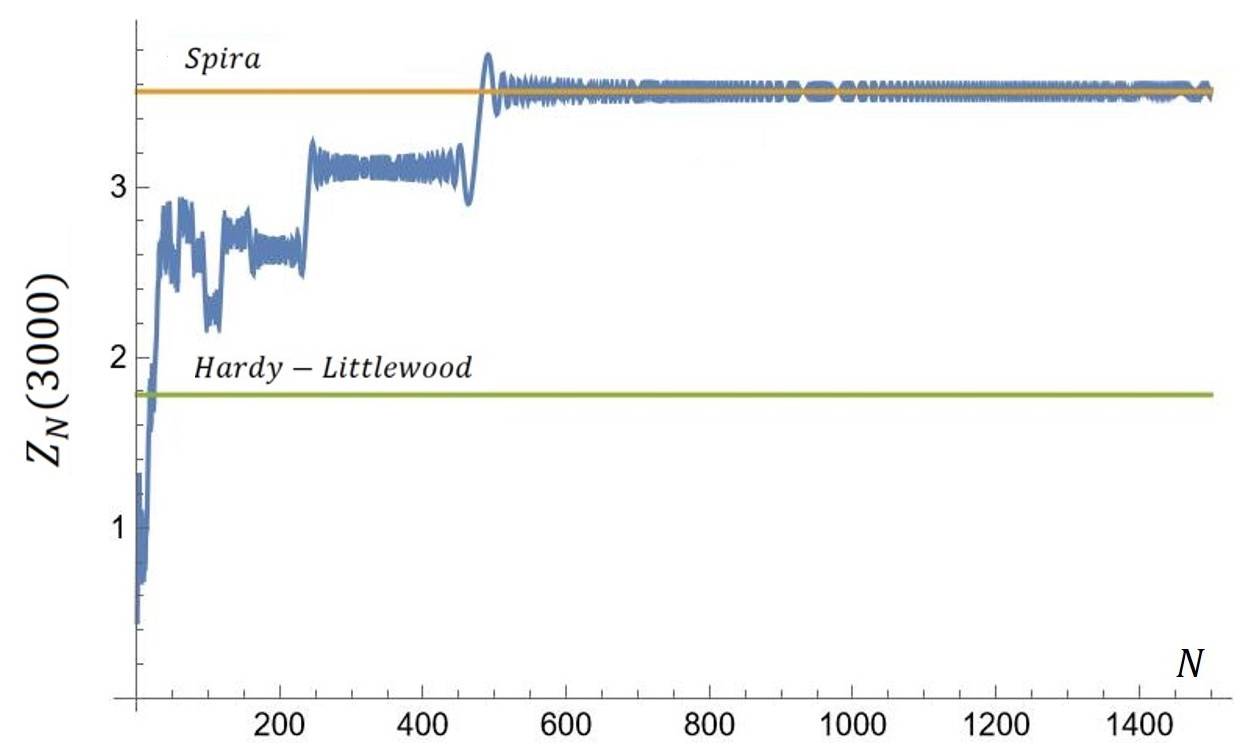}
\caption{\small Values of the sections $Z_N(t)$ for the fixed value of $t=3000$ and $N=1,...,1500$ (blue), the Hardy-Littlewood approximation of $\frac{1}{2}Z(3000)$ (green) and the Spira approximation $Z(3000)$ (brown).}
\label{fig:f1.0}
\end{figure}

For numerical computations, employing the Hardy-Littlewood main sum coupled with the Riemann-Siegel formula for evaluating $R(t)$ clearly seems significantly more efficient than utilizing Spira's approximation. For instance, for the following reasons: 

\begin{enumerate}
\item The number of terms required for Spira's approximation, $N(t) = \left [ \frac{t}{2} \right ]$, increases quadratically compared to $\widetilde{N}(t) = \left [ \sqrt{\frac{t}{2\pi}} \right ]$ required by the Hardy-Littlewood method, making it far more costly for numerical calculations. This characteristic alone likely made the approximation \eqref{eq:Spira}, even if folklorically known, impractical for computational use, particularly before the advent of computers. 

\item The theoretical asymptotic estimation of Spira's approximation error, obtained by classical means, is $O\left(\frac{1}{\sqrt[4]{t}}\right)$, which similar to that of the Hardy-Littlewood formula, before the application of the Riemann-Siegel expansion of the error whose first term already gives an error of order $O \left(t^{-\frac{3}{4}} \right )$. It should be noted, however, that these are approximate results, and hence this seemingly superior asymptotic bound does not necessarily ensure greater practical accuracy over Spira's approximation.

\end{enumerate}

\subsection{Spira's conjecture and the Absence of Theoretical Justification} Despite being far more costly, Spira, through his empirical investigations, discovered that approximation \eqref{eq:Spira} might have a critical significance, see S.6 of \cite{SP2}. Remarkably, contrary to the Hardy-Littlewood formula, this higher-range approximation does not seem to admit zeros off the real line. That is, it is sensitive enough to observe the RH without the need for the expansion of the error term, as in the classical Riemann-Siegel formula. This observation, although not explicitly stated by Spira in the following form, can be formalized as follows:  

\begin{conj}[Spira's RH for sections] All the non-trivial zeros of $Z_{N(t)}(t) $ are real. 
\end{conj} 

Figure \ref{fig:f1.1} shows $ln \abs{Z(t)}$ (orange) and comparing between the Spira approximation $ln \abs{Z_N(t)}$ with $N=N(t)=205$ (blue - left) and the classical Hardy-Littlewood approximation $ln \abs{2 Z_N(t)}$ with $N=\widetilde{N}(t)=8$ (blue - right) in the range $412<t<419$:  

\begin{figure}[ht!]
\centering
\includegraphics[scale=0.35]{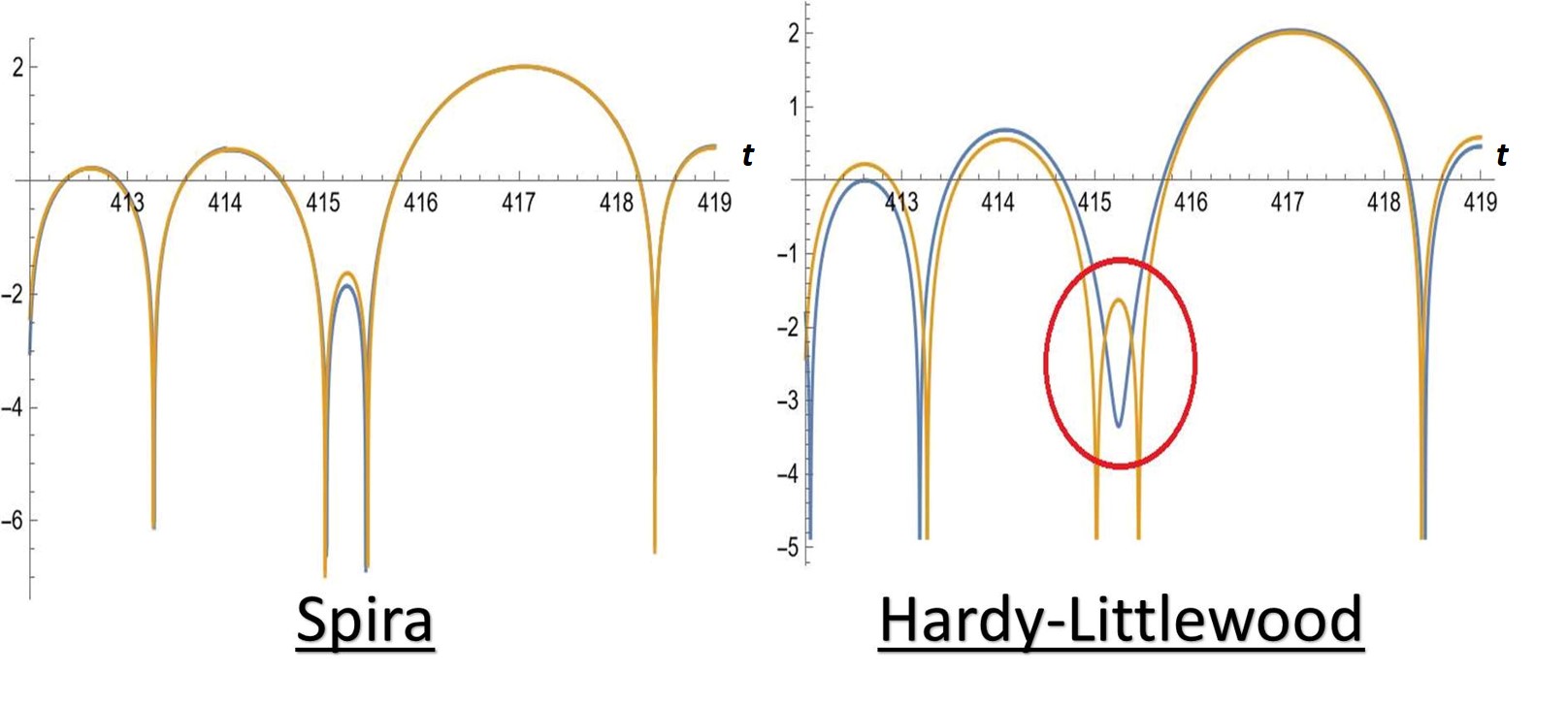}
\caption{\small Graphs of $ln \abs{Z(t)}$ (orange) and the Spira approximation $ln \abs{Z_N(t)}$ with $N=N(t)=205$ (blue - left) and the classical Hardy-Littlewood approximation $ln \abs{2 Z_N(t)}$ with $N=\widetilde{N}(t)=8$ (blue - right) in the range $412<t<419$}
\label{fig:f1.1}
\end{figure}

Figure \ref{fig:f1.1} illustrates an instance of two real zeros accurately predicted by Spira's formula, as anticipated, but missed by the Hardy-Littlewood formula (highlighted here in red), which actually erroneously identifies them as complex zeros

Spira's conjecture, while grounded in empirical and experimental observations, lacks direct theoretical underpinnings. Our aim in this work is to present a theoretical justification for the phenomena observed in Spira's conjecture of the advantage of the higher order section $Z_{N(t)}(t) $ over the Hardy-Littlewood sections $Z_{\widetilde{N}(t)}(t)$, via our new techniques of accelerated approximations.

\section{Spira's Approximation and Accelerated Approximations} 

 In \cite{J} we have developed an AFE based on the accelerated global series for $\zeta(s)$ due to Hasse-Sondow, which admits an error of exponential decay. For the $Z$-function we have the following variant of our accelerated formula:
\begin{equation} 
\label{eq:acc}
Z(t) =\widetilde{Z}_{N(t)}(t)+ O \left ( e^{- \omega t} \right ), 
\end{equation} 
where
\begin{equation} 
\widetilde{Z}_N(t):= \sum_{n=0}^{N} \frac{1}{2^{n+1}} \sum_{k=0}^n \binom{n}{k} \frac{cos( \theta(t)- ln(k+1) t)}{\sqrt{k+1}} ,
\end{equation} 
with $\omega>0$ a certain positive constant computed in \cite{J} and $N(t)= \left [\frac{t}{2} \right ]$ is the same order as in Spira's approximation. The following Fig. \ref{fig:f1.2} shows $ln \abs{Z(t)}$ (orange) and our approximation $ln \abs{ \widetilde{Z}_{N(t)}(t)}$ with $N=N(t)=205$ (blue) in the range $412<t<419$.:

\begin{figure}[ht!]
\centering
\includegraphics[scale=0.45]{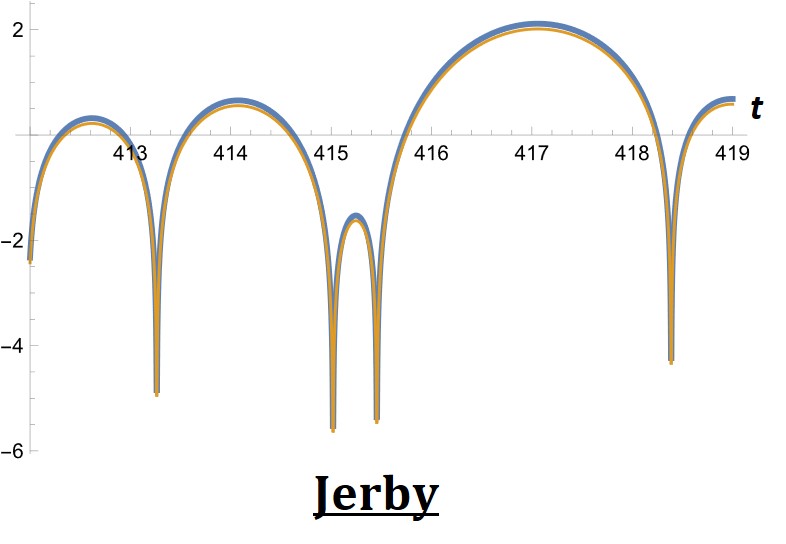}
\caption{\small Graphs of $ln \abs{Z(t)}$ (orange) and our accelerated approximation $ln \abs{ \widetilde{Z}_{N(t)}(t)}$ of \cite{J} with $N=N(t)=205$ (blue) in the range $412<t<419$.}
\label{fig:f1.2}
\end{figure}

Note that our approximation via the accelerated $\widetilde{Z}_{N(t)}(t)$ achieves the superior accuracy expected to be attained by the Hardy-Littlewood AFE coupled with the Riemann-Siegel formula when evaluated at the optimal order, according to the Berry and Keating conjecture. Our main result is the following reformulation: 
\begin{prop} \label{prop:acc} The following holds 
\begin{equation}
\label{eq:acc-re} 
Z(t) = \widetilde{Z}_{N(t)}(t) + O \left ( e^{- \omega t} \right ),
\end{equation}
where 
\begin{equation} 
\widetilde{Z}_{N(t)}(t) = \sum_{k=1}^{N(t)} \widetilde{\alpha}^{acc}_k(t) \frac{\cos( \theta(t)- \ln(k) t)}{\sqrt{k}}
\end{equation} 
is the accelerated $N$-th section with the coefficients $\alpha^{acc}_k(t)$ given by
\begin{equation}
\label{eq:acc-co}
\widetilde{\alpha}^{acc}_k(t) := \sum_{n=k-1}^{N(t)} \frac{1}{2^{n+1}} \binom{n}{k-1}.
\end{equation}
\end{prop} 

\begin{proof} Denote by 
\begin{equation} 
\beta_{n,k}(t):=  \beta^0_{n,k} \frac{cos( \theta(t)- ln(k+1) t)}{\sqrt{k+1}},
\end{equation} 
where 
\begin{equation} 
\beta^0_{n,k}:=\frac{1}{2^{n+1}}\binom{n}{k} 
\end{equation}

The accelerated formula is given as the sum of all the coefficients $\beta_{n,k}(t)$ within the triangle $0 \leq n \leq N(t) $ and $0 \leq k \leq n$. In particular, summing first along the $k$ indices leads to the definition of  
\begin{equation} 
A(n,t) = \sum_{k=0}^n \beta_{n,k}(t).
\end{equation} 
The accelerated formula \eqref{eq:acc} is thus given by the summation of $A(n,t)$ for $0 \leq n \leq N(t)$. Figure \ref{fig:f2.1} illustrates the triangle of coefficients $\beta_{n,k}(t)$ and their possible summation orders:   
\begin{figure}[ht!]
\centering
\includegraphics[scale=0.325]{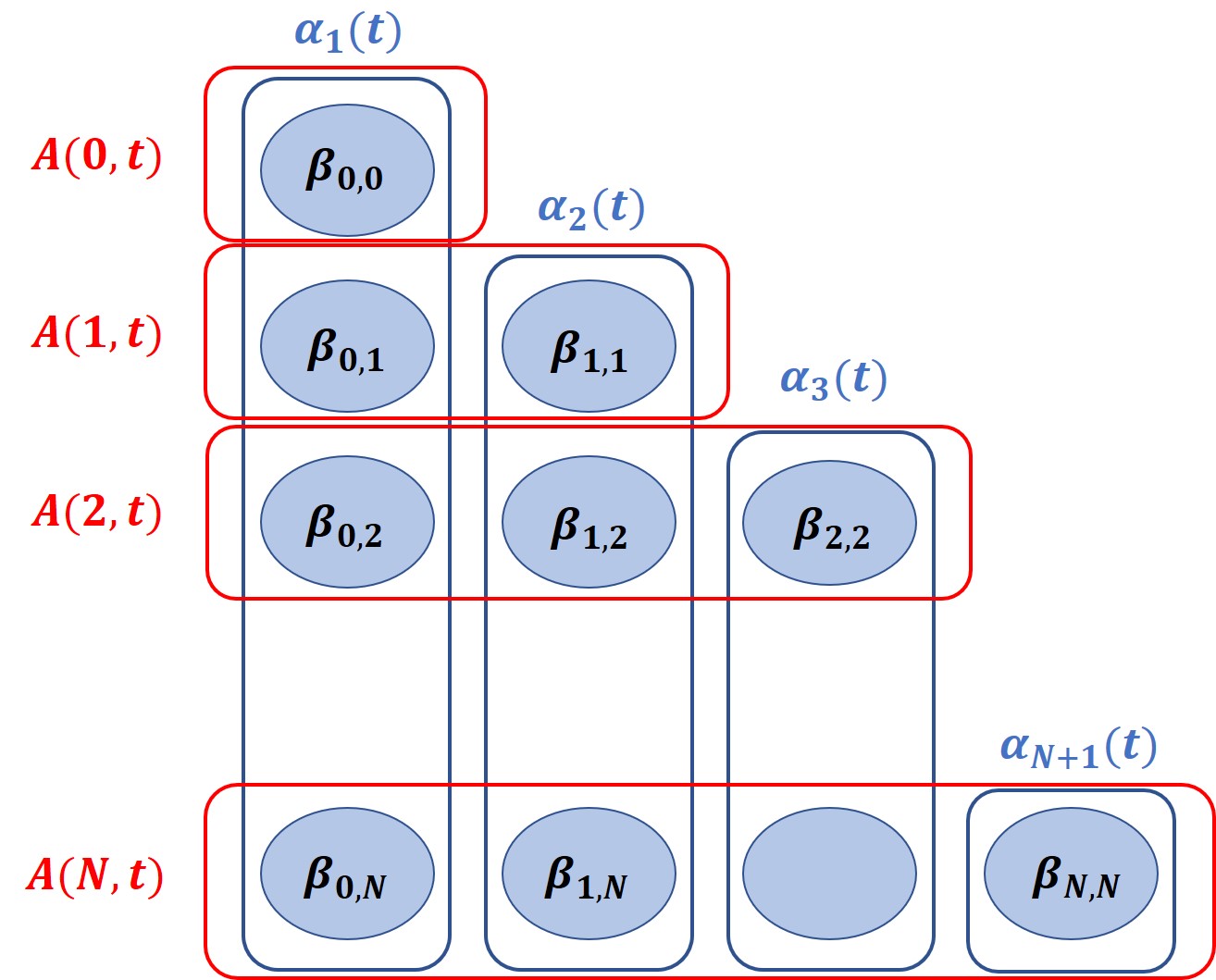}
\caption{\small The triangle of coefficients $\beta_{n,k}(t)$ with $A(n,t)$ arising from the horizontal order of summation (red) and $\widetilde{\alpha}_k(t)$ arising from the vertical order of summation (blue).}
\label{fig:f2.1}
\end{figure}

On the other hand, changing the order of summation to be first along $k-1 \leq n \leq N(t)$ for given $k$, and then along $1 \leq k \leq N(t)$ leads to the definition of $\widetilde{\alpha}_{k}(t)$ and the required formula.  
\end{proof} 

Note that Spira's section can be written as 
\begin{equation}
Z_{N(t)}(t) = \sum_{k=1}^{N(t)} \alpha^{step}_k(t) \frac{\cos( \theta(t)- \ln(k) t)}{\sqrt{k}},
\end{equation}
where the coefficients $\alpha^{step}_k(t)$ are given by the step function:
\begin{equation}
\label{eq:step}
\alpha^{step}_k(t) := \left \{ \begin{array}{cc} 1 & 1 \leq k \leq N(t) \\ 0 & \textrm{otherwise} \end{array} \right..
\end{equation}

Figure \ref{fig:f1.3} shows a comparison between the accelerated coefficients $ \widetilde{\alpha}^{acc}_k(t)$ (blue) and the step coefficients $\alpha^{step}_k(t)$ (orange) of Spira's section for $t=400$ and $k=1,...,400$:   

\begin{figure}[ht!]
\centering
\includegraphics[scale=0.45]{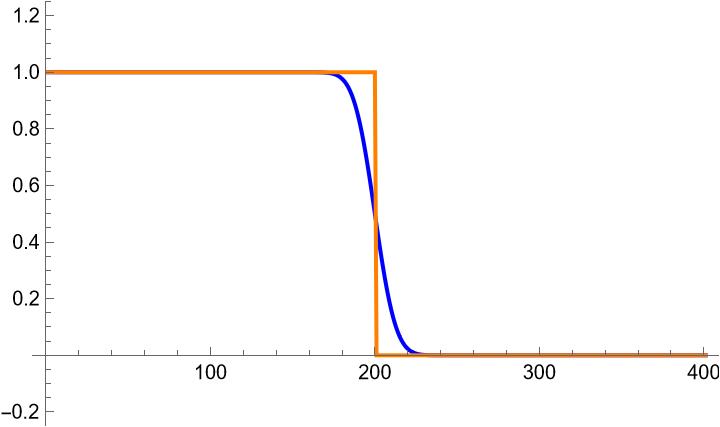}
\caption{\small The accelerated coefficients $ \widetilde{\alpha}^{acc}_k(t)$ (blue) and the step coefficients $\alpha^{step}_k(t)$ (orange) of Spira's section for $t=400$ and $k=1,...,400$.}
\label{fig:f1.3}
\end{figure}

We thus have:

\begin{cor}
\label{cor:2}
The accelerated section $\widetilde{Z}_{N(t)}(t)$ is related to Spira's section $Z_{N(t)}(t)$ via a smoothing of the step coefficients $\alpha^{step}_k(t)$ to obtain $\widetilde{\alpha}^{acc}_k(t)$. Moreover, 
\begin{equation} 
lim_{t \rightarrow \infty} \lvert \lvert \widetilde{\alpha}^{acc}_k(t) - \alpha^{step}_k(t) \rvert \rvert =0
\end{equation} 
in the $\ell_2$ norm.  
\end{cor} 

\section{Discussion and Conclusions} 

Sections of the Hardy $Z$-function are known to approximate the Hardy $Z$-function in two different ways: (a) The classical Hardy-Littlewood AFE given by $2Z_{\widetilde{N}(t)}(t)$ in \eqref{eq:HL} and (b) Spira's approximation given by $Z_{N(t)}(t)$ in \eqref{eq:Spira}. Although the AFE is more efficient and theoretically expected to be superior, Spira observed through his numerical experimentations that in practice his approximation consistently satisfies the RH for any verified $t \in \mathbb{R}$, in the sense that all of its zeros are real. This is in striking contrast to the properties of the AFE, which requires the further evaluation of the Riemann-Siegel formula for this purpose. 

In this work we presented theoretical justification for this mysterious observation, based on our new techniques of the asymptotic analysis of series acceleration, developed in \cite{J} where we introduced a new accelerated approximation satisfying 
\begin{equation} 
Z(t) = \widetilde{Z}_{N(t)}(t) + O (e^{-\omega t}), 
\end{equation} 
where $\omega>0$ is a certain positive constant. This approximation achieves the superior accuracy expected to be attained by the Hardy-Littlewood AFE coupled with the Riemann-Siegel formula when evaluated at the optimal order, according to the Berry and Keating conjecture. In general, for any sequence $\alpha_k = ( \alpha_1,\alpha_2,...) \in \ell_2$ set 
\begin{equation} 
Z(t ; \alpha_k) := \sum_{k=1}^{\infty} \alpha_k \frac{cos(\theta(t)-ln(k)t)}{\sqrt{k}}.
\end{equation}
In a unified manner, we have shown in Proposition \ref{prop:acc} that both Spira's sections and our accelerated sections can be expressed as 
elements in this space 
\begin{equation} 
\begin{array}{ccc} 
\widetilde{Z}_{N(t)}(t) = Z(t ; \widetilde{\alpha}_k^{acc}(t)) & ; Z_{N(t)}(t) = Z(t; \alpha_k^{step}(t)),
\end{array} 
\end{equation} 
where $\alpha_k^{step}(t)$ is defined in \eqref{eq:step} and the definition of $\alpha_k^{acc}(t)$ is defined in \eqref{eq:acc-co}. Furthermore, Corollary \ref{cor:2} shows that  
\begin{equation} 
lim_{t \rightarrow \infty} \lvert \lvert \widetilde{\alpha}^{acc}_k(t) - \alpha^{step}_k(t) \rvert \rvert =0.
\end{equation} 
This essentially implies that Spira's section $Z_{N(t)}(t)$ asymptotically coincides with our accelerated sections $\widetilde{Z}_{N(t)}(t)$ as $ t \rightarrow \infty$. 

In conclusion, our results provide the required theoretical justification to the numerical phenomena observed by Spira. Indeed, the fact that the difference between the accelerated sections $\widetilde{Z}_{N(t)}$ and the Hardy $Z$-function $Z(t)$ itself is imperceptible, 
coupled with the convergence of $Z_{N(t)}(t)$ to $\widetilde{Z}_{N(t)}$ at large \(t\), establishes that Spira's RH for sections is essentially equivalent to RH for $Z(t)$ itself. Moreover, our results further motivate the definition of the new parametrized space of sections $Z(t; \alpha_k)$ and the study of the properties of their zeros with respect to variation of the parameters $\alpha_k \in \ell_2$,  suggesting a promising direction for further exploration.

\section*{Declarations}

\subsection*{Funding}
No funding was received to assist with the preparation of this manuscript.

\subsection*{Conflicts of Interest/Competing Interests}
The authors declare that they have no conflict of interest or competing interests relevant to the content of this article.

\subsection*{Data Availability}
The authors declare that the data supporting the findings of this study are available within the paper or from the corresponding author upon reasonable request.

\bibliographystyle{plain} 

\end{document}